\documentclass{article}   	
\usepackage{geometry}               
\usepackage{graphicx,color}	
						
\usepackage[dvipsnames]{xcolor}
\usepackage{amssymb,amsmath,amsthm,amsfonts}
\usepackage[english]{babel}
\usepackage[utf8]{inputenc}

\usepackage[pdfencoding=auto,colorlinks]{hyperref}
\hypersetup{linkcolor=blue,citecolor=blue,filecolor=black,urlcolor=black}

\usepackage[sc]{mathpazo}

\newtheorem{proposition}{Proposition}[section]
\newtheorem{theorem}[proposition]{Theorem}

\newtheorem{lemma}[proposition]{Lemma}

\theoremstyle{definition}

\theoremstyle{remark}
\newtheorem{remark}[proposition]{Remark}
\numberwithin{equation}{section}

\newcommand{\dsum}{\displaystyle \sum}
\newcommand{\N}{{\mathbb{N}}}

\newcommand{\R}{{\mathbb{R}}}

\newcommand{\loc}{{\mathrm{loc}}}

\newcommand{\Bcal}{{\mathcal{B}}}

\def\sideremark#1{\ifvmode\leavevmode\fi\vadjust{\vbox to0pt{\vss
 \hbox to 0pt{\hskip\hsize\hskip1em
 \vbox{\hsize2.1cm\tiny\raggedright\pretolerance10000
  \noindent #1\hfill}\hss}\vbox to15pt{\vfil}\vss}}}%

\title{On the emergence of dead cores in elliptic systems with sublinear competitive interactions}
\author{Hugo Tavares, Gianmaria Verzini, Junwei Yu}

\begin{document}

\maketitle

\begin{abstract}
In this paper, we study semilinear elliptic systems with sublinear coupling terms, showing that, under competition-type interactions, solutions typically have dead cores, i.e., they vanish on certain open subsets of the domain. We apply our results to a large class of solutions treated in the literature, for instance to ground states and least energy sign-changing solutions, under Dirichlet boundary conditions or in the whole space. The proofs are based on a general result stating that subsolutions to a certain sublinear equation have dead cores, and on uniform H\"older bounds.
\end{abstract}%
\noindent
{\footnotesize \textbf{AMS-Subject Classification}}
{\footnotesize 35J47, 35B45, 35B51.}\\
{\footnotesize \textbf{Keywords}}  
{\footnotesize Comparison principle. Solutions with dead core. Sublinear couplings. Gradient systems with competitive interactions. Semilinear elliptic systems.}

\section{Introduction and main results}
Consider the following semilinear elliptic system with $k\ge 2$ competing densities $u_i$:
\begin{equation}\label{eq:gen_sys}
-\Delta u_i=f_i(x,u_i)-\beta |u_i|^{p-1} u_i \dsum_{j\neq i}a_{ij} |u_j|^{p+1} \quad \text{ in }\Omega,  \qquad  i=1,\ldots, k,
\end{equation}
where $\Omega$ is a domain in $\R^N$ (possibly the whole space), $N\geq 1$, the functions 
$f_i$ are suitable  nonlinearities, and the parameters $p>0$, $\beta>0$ and $a_{ij}=a_{ji}>0$ ($i\neq j$) 
encode the competitive nature of the mutual interactions. Notice that the competition terms have global 
homogeneity of degree $ 2p+1>1$, thus they are always superlinear in the pair $(u_i,u_j)$, but they may be 
sublinear in the single density $u_i$, provided $p<1$. 

This system has been treated in several different contexts, for instance when $\Omega=\R^N$ or when 
$\Omega$ is a smooth bounded domain, under Dirichlet or Neumann boundary conditions, 
with nonlinearities like 
\[
f_i(x,s)= \mu_i |s|^{q-1}s + (\lambda_i - V_i(x)) s,
\] 
in the (globally) superlinear, Sobolev--subcritical and critical cases
\begin{equation}\label{eq:subcrit}
1<q\le 2^*-1,\qquad 0<p\le 2^*/2-1.
\end{equation}
Not being feasible to quote all topics and references here, without being exhaustive we highlight only that 
the results in the literature span from the existence and multiplicity of nonnegative or sign-changing, fully nontrivial solutions (i.e., with all components $u_i$ being non trivial), with $f_i$ given (see for instance \cite{WeiWeth_ARMA2008,MR2592975,MR2629888,ChenZou_ARMA2012,MauroSchieraTavares_JDE2016,ClappPistoia_CalcVar_2018,ClappSzulkin_NoDEA_2019,LiWeiWu_JDE_2024}) or under $L^2$-constraints \cite{MR3539467,MR3639521,MR3777573,MR3918087,LiLiuZou_JDE_2026}, to the asymptotic behavior of solutions as $\beta\to \infty$, see \cite{MR2599456,DancerWangZhang_JFA_2012,SoaveZilio_ARMA2015}. For other results and references, we refer to the cited bibliography.

\smallbreak

For families of nonnegative solutions, when $p\geq 1$, one can apply the strong maximum principle to (regular) nonnegative solutions
\[
u_\beta = (u_{1,\beta},\dots, u_{k,\beta})
\] 
of \eqref{eq:gen_sys} to  obtain that, for each $i$, either $u_{i,\beta}\equiv0$ or   $u_{i,\beta}>0$ in $\Omega$. On the other hand, when $0<p<1$, the strong maximum principle 
cannot be used in its standard form,  thus one may wonder if solutions with \emph{dead core}, i.e. 
which vanish on a set with nonempty interior, may exist.  In the literature, this question has been 
generally overlooked. Indeed, notice that \eqref{eq:subcrit} allows for values of $p$ smaller than 1 (actually, this is necessary, for $N\geq 5$); on the other hand, 
several papers show existence of nonnegative, fully nontrivial solutions, stating (without proof) that these solutions are 
strictly positive. Our paper shows that, if $\beta$ is sufficiently large, such fully nontrivial, nonnegative solutions actually have a dead core. 

These  questions are also related with \cite[Theorem 1.1]{MR4831233}, which proves, using a unique continuation 
property, that the \emph{common} nodal set $\cap_{i} \{u_{i,\beta}=0\}$ cannot have a point of infinite vanishing 
order. In particular, solutions with components exhibiting a common dead core cannot exist. This result is somehow 
natural since, as we already mentioned, the homogeneity of the interaction terms $(s,t)
\mapsto |s|^{p-1}s|t|^{p+1}$ is $ 2p+1>1$.  

Recently, however, the construction of  {infinitely many families of solutions} with dead cores, for a system like \eqref{eq:gen_sys}, was provided in \cite{MR5032422} by means of Lyapunov-Schmidt reduction techniques. This is shown  in the case of $k=2$ components, radial potentials with specific behavior at infinity, and 
power nonlinearities, in the whole $\R^N$, and is, up to our knowledge, the first existence result which explicitly mentions dead core solutions for systems like \eqref{eq:gen_sys}.\medskip

The main aim of this paper is to show that, for system  \eqref{eq:gen_sys} with $0<p<1$, the existence of dead core solutions \emph{is the typical situation}, at least for large $\beta>0$, i.e., when the competition becomes the prevailing phenomena, as long as global vanishing does not occur. In our main result, we deal with weak solutions (possibly signed) of \eqref{eq:gen_sys}, for 
any choice of the Caratheodory functions $f_i$ such that:
\begin{equation}\label{eq:ass_f}
\text{for some $A>0$,}\qquad  
\sup \left\{\frac{|f_i(x,s)|}{|s|^p} : x\in\Omega,\  |s| \le A,\ i=1,\dots,k\right\}=:L_A<+\infty.
\end{equation}
\begin{theorem}\label{thm:main1}
Let  $\Omega\subset\R^N$ ($N\ge 1$) be open, $a_{ij}=a_{ji}>0$ ($i\neq j$), $0<p<1$, and assume that $(f_1,\ldots, f_k)$ satisfies \eqref{eq:ass_f} for some $A>0$. 

Let $(u_{\beta})_\beta\subset H^1_{\loc}(\Omega,\R^k)$ be a family of solutions of \eqref{eq:gen_sys} such that, for every $\beta>0$:
\begin{enumerate}
\item\label{i:1} $\|u_\beta\|_{L^\infty(\Omega)}\leq A$;
\item\label{i:2} $\|u_{h,\beta}\|_{L^\infty(\Omega')}\geq \bar \delta$, for some $1 \le h \le k$, $\Omega'\Subset \Omega$ open and relatively compact, and $\bar \delta\in (0,A)$.
\end{enumerate}
Then, for every $0<\delta<\bar\delta$, there exist $\bar \beta,R>0$ such that, for every $\beta>\bar \beta$,
\[
|u_{h,\beta}|>\delta,\qquad u_{i,\beta}\equiv0\text{ for every }i\neq h,\qquad \text{in } 
B_R(x_\beta),
\]
for suitable $x_\beta\in \Omega'$. 
In particular, all components except possibly one have a dead core containing a ball with a fixed radius.
\end{theorem}

The proof of this theorem is rather straightforward, and is structured as follows. First, we provide a 
general result, of independent interest, which asserts that solutions of certain sublinear equations have 
dead cores (Proposition \ref{prop:deadcore}); then, we recall a key result from \cite{MR3485154}, which 
states that uniform $L^\infty$ bounds in $\Omega$ imply (in any compactly contained subset) uniform 
H\"older bounds (see Theorem \ref{thm:sttz} ahead). The latter, combined with assumption \ref{i:2}, 
directly implies that the support of $u_h$ contains a ball of fixed radius, while the 
former, together with Kato's inequality, assumption \ref{i:1} and system \eqref{eq:gen_sys}, implies the 
existence of dead cores for every $u_i$, $i\neq h$. We point out that the requirement for $\beta$ to be large is, in general, optimal - see Remark \ref{rmk:sharp_synchro} for more details.

Theorem \ref{thm:main1} holds true with no regularity or boundedness assumptions 
on the open set $\Omega$, as well as without boundary conditions on the solutions. The main 
drawback of this approach is that, while the $L^\infty$ bounds in condition \ref{i:1} are assumed in 
$\Omega$, the persistence condition in assumption \ref{i:2} is needed in a relatively compact 
$\Omega'\Subset \Omega$; consequently, one needs to exclude, along the family, the formation of boundary 
layers associated with vanishing in the interior. Of course, this can be more or less difficult to be 
checked, depending on the further information available about the family of solutions. A possible workaround for this 
issue is to prescribe Dirichlet boundary conditions on the boundary of a bounded, regular
$\Omega$ ($C^{1}$ is enough). Indeed, in such a case, it is sufficient (and easier) to check the persistence condition 
on the whole $\Omega$. For instance, being $\Omega$ bounded, this would follow from $L^q$ bounds from below, $q\ge1$. Accordingly, we have the following alternative version of our main result. 

\begin{theorem}\label{thm:main2} Let  $\Omega\subset\R^N$ ($N\ge 1$) be a bounded regular domain, $a_{ij}=a_{ji}>0$ 
($i\neq j$), $0<p<1$, and assume that $(f_1,\ldots, f_k)$ satisfies \eqref{eq:ass_f} for some $A>0$. 

Let $(u_{\beta})_\beta\subset H^1_{0}(\Omega,\R^k)$ be a family of solutions of \eqref{eq:gen_sys} such that, for every $\beta>0$:
\begin{enumerate}
\item\label{iDir:1} $\|u_\beta\|_{L^\infty(\Omega)}\leq A$;
\item\label{iDir:2} $\|u_{h,\beta}\|_{L^\infty(\Omega)}\geq \bar \delta$, for some $1 \le h \le k$, and $\bar \delta\in (0,A)$.
\end{enumerate}
Then the conclusions of Theorem \ref{thm:main1} hold.
\end{theorem}

Some remarks are in order.
\begin{remark}\label{rmk:2indices}
If assumption \ref{i:2} above holds for two indices $h_1\neq h_2$, then of course the theorems 
apply twice, and every component of the vector solution has a dead core containing a ball with a fixed radius, for $\beta$ large enough.
\end{remark}
\begin{remark}\label{rmk:weak-strong}
 Notice that, by elliptic regularity and assumption \eqref{eq:ass_f}, any weak solution of \eqref{eq:gen_sys} 
belonging to $H^1_\loc(\Omega)\cap L^\infty(\Omega)$ is $C^{1,\alpha}_\loc(\Omega)$, for every $\alpha<1$. On the other hand, the $L^\infty$ bounds can be obtained by suitable growth assumptions on the nonlinearities $f_i$, see Section \ref{sec:applications} ahead.
\end{remark}
\begin{remark}\label{rmk:anyfamily}
Although, for simplicity, we state the theorem for families of solutions indexed on $\beta \in 
(0,+\infty)$, the same result follows also for families with $\beta \in \Bcal$, as long as $+\infty$ is 
an accumulation point of the parameter set $\Bcal$. In particular, to obtain solutions with $\beta$ large and dead cores, 
it is enough to require that assumptions \ref{i:1}, \ref{i:2} hold for some sequence $\beta_n\to+\infty$. 
\end{remark}
\begin{remark}\label{rmk:sharp_synchro}
To obtain dead cores, we require the competition parameter to be large enough. On the other hand, the 
solutions in \cite{MR5032422} have dead cores for every fixed $\beta>0$, as long as each component has a 
sufficiently large number of peaks, sufficiently far from  the origin. Then one may wonder if 
\emph{every} solution with $\beta>0$ has dead cores. This is not the case as, for instance,
\[
(u_\beta,v_\beta)=\left(\beta^{-\frac{1}{2p}},\beta^{-\frac{1}{2p}}\right),\qquad \beta> 0
\] 
solves
\[
\begin{cases}
-\Delta u= u-\beta u^{p} v^{1+p} & \text{in } \Omega, \\ 
-\Delta v= v-\beta v^{p} u^{1+p} & \text{in } \Omega. \\ 
\end{cases}
\]
This also shows that, in general, our persistence condition \ref{i:2} cannot be removed. 
One can build similar examples, in the realm of synchronized solutions, also with Dirichlet boundary conditions. 
\end{remark}

As we already mentioned, the construction of a particular family of solutions with dead cores, for a 
specific instance of \eqref{eq:gen_sys}, was recently provided in \cite{MR5032422} by means of Lyapunov-
Schmidt reduction techniques. As an application of our results, here we are going to show that several 
other families of solutions, known in the literature, satisfy our assumptions and hence exhibit dead cores. In 
particular, this is true for every type of ground state solutions of Schr\"odinger-type systems with 
competition, such as least energy normalized solutions (with fixed masses) and least action 
ones (with fixed frequencies), and for the nodal ground states as well. All these examples are 
characterized by the fact that the families under investigation have a variational characterization, so that 
uniform estimates of the energy/action levels are available. Then, regularity arguments based on 
iteration schemes of Moser/Brezis-Kato type (see Lemmas \ref{lemma:uniform_bounds} and
\ref{lemma:Linfty_critical} ahead) allow to infer uniform $L^\infty$-bounds, thus triggering 
the application of Theorems \ref{thm:main1} and \ref{thm:main2}. On the other hand, also families of 
solutions obtained by Lyapunov-Schmidt reduction, as those in \cite{MR5032422}, may enjoy uniform 
$L^\infty$-bounds, depending on the convergence to suitable limit profiles.

The paper is structured as follows: in Section \ref{sec:applications} we illustrate the above mentioned 
applications, based on the $L^\infty$-bounds provided by Lemmas \ref{lemma:uniform_bounds} and
\ref{lemma:Linfty_critical} therein; Section \ref{sec:proof_of_thms} contains the proof of Theorems \ref{thm:main1} and
\ref{thm:main2}, while Section \ref{sec:proof_of_lemmas} is devoted to that of Lemmas \ref{lemma:uniform_bounds} and
\ref{lemma:Linfty_critical}.

\section{Applications}\label{sec:applications}

In this section we provide some applications of our main results, to support our claim that the formation of dead 
cores is the typical phenomenon, for solutions of \eqref{eq:gen_sys} with $0<p<1$, when $\beta$ is large enough. More 
precisely, we will describe in detail an example, under specific choices of the nonlinearities $f_i$ and of the 
families of solutions $(u_\beta)_\beta$. We also mention some papers in the literature to which one may 
apply such result or suitable variations. It will be clear that our arguments can be 
adjusted to deal also with different cases.

The common idea behind these applications is to consider families of solutions with a variational characterization, 
and to translate uniform (in $\beta$) estimates of the energy/action levels into uniform $L^\infty$ ones. To this 
aim, we first state two results, concerning uniform $L^\infty$ bounds of nonnegative subsolutions of certain 
equations, with either subcritical or critical growth. While these results are more or less standard, we did not 
find a reference which would fit directly in our purposes. For that reason, for completeness, we present their 
proofs in Section \ref{sec:proof_of_lemmas}. For concreteness, apart for some remarks, we bound ourselves to the case when 
\eqref{eq:gen_sys} is complemented with homogeneous Dirichlet boundary conditions, in dimension $N\ge 3$. 

\begin{lemma}[$L^\infty$ bounds: subcritical case]\label{lemma:uniform_bounds}
Let $\Omega$ be an open subset of $\R^N$, with $N\geq 3$. Take $a>0$, $M>0$, $q\in (1,2^*-1)$, and let $u\in H^1_0(\Omega)$ be a solution to 
\[
\begin{cases}
-\Delta u \leq a(u+u^{q}) &\text{ in } \Omega,\\
u\geq 0 & \text{ in } \Omega,
\end{cases}
\]
satisfying $\|u\|_{H^1(\Omega)}\leq M$. Then there exists $C=C(a,M,N,q)>0$ such that 
\[
\|u\|_{\infty} \leq C.
\]
\end{lemma}

\begin{lemma}[$L^\infty$ bounds: critical case]\label{lemma:Linfty_critical}
Let $\Omega$ be an open subset of $\R^N$, with $N\geq 3$. Take $a>0$, and let $(u_n)_n\subset  H^1_0(\Omega)$ be a family of solutions satisfying
\[
\begin{cases}
-\Delta u_n \leq a(u_n+u_n^{2^*-1}) &\text{ in } \Omega,\\
u_n\geq 0 & \text{ in } \Omega.
\end{cases}
\]
Assume also the existence of $u_\infty \in H^1_0(\Omega)$ such that $u_n\to u_\infty$ in $L^{2^*}(\Omega)$. Then there exists $C=C(a,N,u_\infty)>0$ and $\bar n$ such that 
\[
\|u_n\|_{\infty} \leq C(\|u_n\|_{2}+\|u_n\|_{2^*}^{2^*}) \quad \text{ for every } n\geq \bar n.
\]
\end{lemma}

We point out that, in the two previous lemma, one can take $\Omega=\R^N$.
\begin{remark}\label{rmk:Linfty_bds_w/out_boundary}
Assume $u\geq 0$ solves $-\Delta u\leq C(u+u^q)$ in $\Omega$, with $q< 2^*-1$, and no information about $u$ on $\partial 
\Omega$ is known. For every $\Omega'\Subset \Omega$, we can take a nonnegative cutoff function $\varphi\in 
C^\infty(\Omega)$ with $\varphi=1$ in $\Omega'$, and we could then repeat the arguments of the proofs of Lemmas 
\ref{lemma:uniform_bounds} with the only difference therein being that, at the beginning, we should use  
$u^{2s+1}\varphi^2$ instead of $u^{2s+1}$ as test function (see the proof of \cite[Lemma B.3]{Struwe_book} 
for the details). This yields a uniform bound of $\|u\|_{L^\infty(\Omega')}$, exactly as before. An analogous remark can be done regarding Lemma \ref{lemma:Linfty_critical}.
\end{remark}

In the following, we focus on a version of \eqref{eq:gen_sys} with fixed nonlinearities of polynomial 
type on a smooth bounded  domain $\Omega\subset \R^N$. For $0<p\leq 2^*/2-1$, let us consider the system
\begin{equation}\label{eq:sys_action}
\begin{cases}
-\Delta u+\lambda_1 u=\mu_1 |u|^{2p}u-\beta |u|^{p-1}u |v|^{p+1} & \text{in }\Omega, \\ 
-\Delta v+\lambda_2 v=\mu_2 |v|^{2 p}v-\beta |v|^{p-1}v |u|^{p+1} & \text{in }\Omega, \\ 
u=v=0 \text { on } \partial \Omega, & \end{cases}
\end{equation}
where $\mu_1,\mu_2>0$, $-\lambda_1(\Omega)<\lambda_1,\lambda_2$ if $p<2^*/2-1$, and 
$-\lambda_1(\Omega)<\lambda_1,\lambda_2<0$ if $p=2^*/2-1$, whose solutions are critical points in $
H^1_0(\Omega)$ of the action functional
\begin{multline*}
J_\beta(u,v):=\frac12\int_\Omega(|\nabla u|^2 + \lambda_1 u^2 +|\nabla v|^2 + \lambda_2 v^2 )\\
-\frac1{2p+2}\int_\Omega(\mu_1| u|^{2p+2} + \mu_2| v|^{2p+2} )
+\frac\beta{p+1}\int_\Omega|u|^{p+1}| v|^{p+1}.
\end{multline*}

\begin{proposition}\label{prop:action}
Let $(u_\beta,v_\beta)_\beta$  be a family of fully nontrivial solutions of \eqref{eq:sys_action} 
such that, for a constant $C>0$ independent of $\beta$ and for every $\beta\ge1$,
\begin{equation}\label{eq:bound_J_beta_ind}
J_\beta(u_\beta,v_\beta)\le C.
\end{equation}
Assume moreover that, when $p=2^*/2-1$, $(u_\beta,v_\beta)\to(u_\infty,v_\infty)
$, strongly in $H^1$, as $\beta\to+\infty$ (possibly up to subsequences). 

Then $(u_\beta,v_\beta)_\beta$ satisfies assumptions \ref{iDir:1}, \ref{iDir:2} in Theorem \ref{thm:main2} 
(possibly up to subsequences) and, in particular, both $u_\beta$ and $v_\beta$ have a dead core for $\beta$ large enough.
\end{proposition}
\begin{proof}
Testing the equations with $u_\beta$, $v_\beta$, we obtain the validity of the Nehari-type conditions
\[
\begin{split}
\int_\Omega\mu_1| u|^{2p+2} = 
\int_\Omega(|\nabla u|^2 + \lambda_1 u^2 + \beta|u|^{p+1}| v|^{p+1} ),\\
\int_\Omega\mu_2| v|^{2p+2} = 
\int_\Omega(|\nabla v|^2 + \lambda_2 v^2 + \beta|u|^{p+1}| v|^{p+1} ).
\end{split}
\]
Substituting into \eqref{eq:bound_J_beta_ind} one immediately infers that $(u_\beta,v_\beta)_\beta$ is uniformly 
bounded in $H^1_0(\Omega)$. Moreover, using Kato's inequality (see e.g. 
\cite[Corollary 1.2]{MR2056467}) we obtain
\[
-\Delta |u_{\beta}| \le -\lambda_1 |u_{\beta}| + \mu_1|u_{\beta}|^{2p+1} -\beta |u_{\beta}|^p |v_{\beta}|^{p+1}
\le  |\lambda_1| |u_{\beta}| + \mu_1|u_{\beta}|^{2p+1},
\]
and a similar inequality holds for $v_\beta$. Then, we can apply either Lemma \ref{lemma:uniform_bounds} or 
Lemma \ref{lemma:Linfty_critical}, to both $|u_\beta|$ and $|v_\beta|$, obtaining that assumption \ref{iDir:1} 
in Theorem \ref{thm:main2} is satisfied.

To check \ref{iDir:2}, we use the Sobolev embedding and the Nehari-type identity to write
\[
\|u_\beta\|_{L^{2p+2}}^{2} \le C \|u_\beta\|_{H^1}^{2} \le C' \|u_\beta\|_{L^{2p+2}}^{2p+2},
\]
whence, recalling that $(u_\beta,v_\beta)$ is fully nontrivial and $\Omega$ is bounded,
\[
0<C''\le \|u_\beta\|_{L^{2p+2}}^{2p} \le C'''\|u_\beta\|_{L^{\infty}}^{2p},
\]
and a similar inequality holds for $v_\beta$.
\end{proof}

Thanks to this result, the existence of solutions with dead cores for \eqref{eq:sys_action}, for $\beta$ large, 
boils down to the construction of families of solutions satisfying \eqref{eq:bound_J_beta_ind}. Notice that this is 
always true for solutions which come from a minimax principle associated with $J_\beta$, as long as competitors $(u,v)$ 
with $u\cdot v\equiv 0$ are admissible. In particular, this is true for (fully nontrivial) ground states and 
least energy sign-changing solutions, whenever they exist. Consequently, Proposition \ref{prop:action} 
(or small variants) applies to the solutions obtained in several papers concerning the Sobolev subcritical case, ranging 
for instance from  \cite{ContiTerraciniVerzini_AIHP_2002} to \cite{ClappSzulkin_NoDEA_2019}, but also in the 
more delicate Sobolev critical case: we refer to several different kinds of least energy solutions, 
either nonnegative or sign-changing, found for instance in  \cite[Theorem 1.3]{ChenZou_CalcVar_2015}, 
\cite[Theorem 1.2]{ChenLinZou_CPDE_2014} and \cite[Theorem 1.3]{TavaresYouZou_JFA_2022}. In these 
examples, the further condition concerning the strong $H^1$ convergence 
$(u_\beta,v_\beta)\to(u_\infty,v_\infty)$, as $\beta\to+\infty$, is a consequence e.g. of 
\cite[Theorem 1.7]{TavaresYouZou_JFA_2022}.
\begin{remark}[Systems on $\R^N$, minimax solutions]\label{rmk:unbdd_domains}
The ideas beyond Proposition \ref{prop:action} can be adapted also to systems in the whole space, and/or to families of solutions not necessarily of least energy type. As an example, in \cite{ClappPistoia_CalcVar_2018}, the authors deal with the following system 
\[
\begin{cases}
-\Delta u=\mu_1|u|^{2^*-2} u+\lambda \alpha|u|^{\alpha-2}|v|^\beta u \\
-\Delta v=\mu_2|v|^{2^*-2} v+\lambda \beta|u|^\alpha|v|^{\beta-2} v \\
u, v \in D^{1,2}\left(\mathbb{R}^N\right),
\end{cases}
\]
with $\mu_1,\mu_2,\alpha-1,\beta-1>0$ and $\alpha+\beta=2^*$. For $\Gamma$  a closed subgroup of $O(N+1)$, 
the group of linear isometries of $\mathbb{R}^{N+1}$, the authors show existence of families of solutions 
which are invariant under the conformal action of $\Gamma$ on $\mathbb{R}^N$ induced by the stereographic 
projection $\sigma: \mathbb{S}^N \rightarrow \mathbb{R}^N \cup\{\infty\}$; such solutions are of minimax 
type for $J_\beta$, or, equivalently, minimizers over suitable natural constraints (see 
\cite{ClappPistoia_CalcVar_2018} for more details). Accordingly, uniform bounds on the level of such 
solutions hold true, and, by Lemma \ref{lemma:Linfty_critical}, we infer uniform bounds in $L^\infty(\R^N)$. On the other 
hand, both the persistence condition \ref{i:2} in Theorem \ref{thm:main1}, on suitable 
$\Omega'\Subset\Omega$, and the strong $H^1$ convergence are a consequence of 
\cite[Theorem 1.2]{ClappPistoia_CalcVar_2018}.
\end{remark}
\begin{remark}[Normalized solutions]\label{rmk:normalized}
A result similar to Proposition \ref{prop:action} can be obtained also in the frame of normalized 
solutions, i.e. when \eqref{eq:sys_action} is complemented with the further conditions
\[
\int_\Omega u^{2} = \rho_1>0, \qquad \int_\Omega v^{2} = \rho_2>0,
\]
and $\lambda_i = \lambda_{i,\beta}$ therein are (unknown) Lagrange multipliers. Indeed, in this case,
instead of \eqref{eq:bound_J_beta_ind}, one has to require uniform bounds on the multipliers and on 
the energy
\[
E_\beta(u,v):=\frac12\int_\Omega(|\nabla u|^2  +|\nabla v|^2 )
-\frac1{2p+2}\int_\Omega(\mu_1| u|^{2p+2} + \mu_2| v|^{2p+2} )
+\frac\beta{p+1}\int_\Omega|u|^{p+1}| v|^{p+1}.
\]
Then $L^\infty$ bounds follow similarly as in Proposition \ref{prop:action}, while the persistence
condition is a direct consequence of the $L^2$ normalization. This can be applied, for instance, to 
the global (in the mass subcritical case) or local (in the mass supercritical one) minimizers obtained 
in \cite{MR3918087}, up to the Sobolev critical case (see \cite[Theorem 1.12, Lemma 5.1]{MR3918087} 
for more details). 
\end{remark}

\section{Proof of Theorems \ref{thm:main1} and \ref{thm:main2}}\label{sec:proof_of_thms}

In this section, we provide a proof of our main results. We start with a general result that states that nonnegative subsolutions of a certain sublinear problem have necessarily dead cores. We use the notation $B_{r}=B_r(0)$ to denote the open ball of radius $r$ centered at the origin. Clearly, via a suitable translation, all results hold also for balls centered at other given points.
\begin{proposition}\label{prop:deadcore}
Let $0<p<1$ and assume that, for some $M >0$ and $A>0$, $u\in H^1(B_{2R})$ satisfies (in weak sense)
\[
\begin{cases}
-\Delta u \leq -M u^p \quad &\text{in } B_{2R},\\
u\geq 0 \quad &\text{in } B_{2R},\\
u \leq A  \quad &\text{on } \partial B_{2R}.
\end{cases}
\]
Then
\[
M \geq \frac{(1+p)2^{N(1-p)+p}}{(1-p)^2}\cdot\frac{A^{1-p}}{R^2} 
\qquad\implies\qquad 
u\equiv 0 \ \text{in } B_{R}.
\]
\end{proposition}
To prove the proposition, we use the following comparison principle.
\begin{lemma}\label{lem:comparison}
Let $0<p<1$ and assume that $M>0$ and that $u,v\in H^1(B_{2R})$ satisfy (in the weak sense)
\[
\begin{cases}
- \Delta u \leq - M u^{p} \quad &\text{in } B_{2R},\\
- \Delta v \geq -M v^{p} \quad &\text{in } B_{2R},\\
u,v \geq 0 \quad &\text{in } B_{2R},\\
u \leq v \quad &\text{on } \partial B_{2R}.
\end{cases}
\]
Then
\[
u \leq v \qquad \text{in } B_{2R}.
\]
\end{lemma}
\begin{proof}
Define $w=u-v$, which satisfies $-\Delta w+M(u^p-v^p)\leq 0$ in $B_{2R}$ Testing this inequality with the nonnegative function $w^+\in H^1_0(B_{2R})$, and noticing that $u^p-v^p$ and $w$ have the same sign, we obtain
\[
0\ge \int_{B_{2R}} |\nabla w^+|^2  + M (u^p-v^p)w^+ \ge \int_{B_{2R}} |\nabla w^+|^2,
\]
whence $w\le0$ in $B_{2R}$.
\end{proof}
\begin{proof}[Proof of Proposition \ref{prop:deadcore}]
We consider the auxiliary $C^2$ function
\[
z(t)= t^{\frac{2}{1-p}} 
\qquad\text{ which satisfies },\qquad
\begin{cases}
z''(t)=C_p z^{p}   \quad\text{ in } [0,1]\\
0 \le z(t) \leq 1   \quad\ \ \text{ in } [0,1]\\
z(0)=z'(0)=z''(0)=0,
\end{cases}
\qquad
C_p= \frac{2(1+p)}{(1-p)^2}.
\]
A direct computation shows that
\begin{equation}\label{eq:z'-z}
z'(t) - z(t) =t^{\frac{1+p}{1-p}}\left(\frac{2}{1-p} - t\right)>0
 \qquad\ \ \text{ in } [0,1].
\end{equation}
Using $z$, we define 
\[
v(x)=v(|x|)=v(r)=
\begin{cases}
0   \quad &\text{in } [0,R]\\
\displaystyle h r^{1-N}\cdot z\left(\frac{r}{R}-1\right) &\text{in } [R,2R], 
\end{cases}
\qquad\text{ with }
h = A (2R)^{N-1}.
\]
Then $v\in C^2(\overline{B}_{2R})$, $v\ge0$ in $B_{2R}$ and 
\[
v(2R)= h\cdot (2R)^{1-N} = A \ge \left. u\right|_{\partial B_{2R}}.
\]
Based on Lemma \ref{lem:comparison}, if we show that 
\[
- \Delta v \geq -M v^{p} \qquad \text{ in } B_{2R},
\]
 for $M \geq (1+p)(1-p)^{-2}R^{-2}A^{1-p}2^{N(1-p)+p}$, then $0\leq u\leq v$ in $B_{2R}$ and, in particular, $u\equiv 0$ in $B_R$ and the proposition will follow. This is apparent in $B_R$, 
while $R\le |x-x_0|\le 2R$ implies (below, $z$ and its derivatives are evaluated at $\frac{r}{R}-1$)

\[
\begin{split}
- \Delta v &= -v'' -\frac{N-1}{r} v' = -r^{1-N} \left(r^{N-1} v' \right)'\\
&= -r^{1-N} \left(r^{N-1} \left( \frac{h}{R} r^{1-N}\cdot z'
- (N-1)h r^{-N}\cdot z \right)\right)'
= -r^{1-N} \left(\frac{h}{R} \cdot z'
- \frac{(N-1)h}{r} z \right)'
\\
&= -\frac{h}{R^2} r^{1-N}\cdot z''  
+ (N-1)h r^{1-N}\left(\frac{z'}{Rr} - \frac{z}{r^2}\right)
\end{split}
\]
Using \eqref{eq:z'-z} and $r\ge R$, we infer
\[
\frac{z'}{R} - \frac{z}{r} \ge \frac{1}{R}(z'-z)\ge 0,
\]
whence
\[
\begin{split}
- \Delta v &\ge -\frac{h}{R^2} r^{1-N}\cdot z'' = -\frac{hC_p}{R^2} r^{1-N}\cdot z^p =
-\frac{hC_p}{R^2} r^{1-N}\cdot \left(\frac{r^{N-1}}{h}\right)^p v^p 
 =
-\frac{C_p}{R^2} \cdot \left(\frac{h}{r^{N-1}}\right)^{1-p} v^p\\
&\ge -\frac{C_p}{R^2} \cdot \left(\frac{A(2R)^{N-1}}{R^{N-1}}\right)^{1-p} v^p,
\end{split}
\]
which implies the required property.
\end{proof}

Next, we recall the following fact: under our running 
assumptions, the results in \cite{MR3485154} imply that any family of solutions of 
\eqref{eq:gen_sys}, uniformly bounded in $L^\infty(\Omega)$, with bound independent on $\beta$, enjoys 
uniform H\"older bounds. More precisely, we have the following.
\begin{theorem}[{\cite[Theorems 1.2, 1.3]{MR3485154}}]\label{thm:sttz}
Under the assumptions of Theorem \ref{thm:main1}, for every $\alpha\in(0,1)$ and $\widetilde\Omega\Subset \Omega$, there exists a constant 
$C>0$ such that 
\[
\| u_\beta \|_{C^{0,\alpha}(\widetilde {\Omega})}\le C, \qquad \text{ for every }\beta>0. 
\]

Analogously, under the assumptions of Theorem \ref{thm:main2}, for every $\alpha\in(0,1)$ there is 
$C>0$ such that 
\[
\| u_\beta \|_{C^{0,\alpha}(\overline{\Omega})}\le C, \qquad \text{ for every }\beta>0.  
\]
\end{theorem}
\begin{proof}[Proof of Theorem \ref{thm:main1}] Fix $\widetilde \Omega$ such that $\Omega'\Subset \widetilde \Omega\Subset \Omega$, $\alpha\in (0,1)$, and let $C=C(\widetilde \Omega,\alpha)>0$ be the constant provided by the first part of Theorem \ref{thm:sttz}. Let $\delta<\bar \delta$, and take $x_\beta \in \Omega'$ such that $|u_{l,\beta}(x_\beta)|=\frac{\bar \delta+\delta}{2}>\delta$. For some $R=R(\bar \delta,\alpha,\Omega',\widetilde \Omega)>0$, we have $B_{2R}(x_\beta)\subset \widetilde \Omega$ and
\[
|u_{l,\beta}(x)|\geq |u_{l,\beta}(x_\beta)|-|u_{l,\beta}(x_\beta)-u_{l,\beta}(x)|\geq \frac{\bar \delta+\delta}{2}-C |x-x_\beta|^\alpha>\delta \quad \text{ for } x\in B_{2R}(x_\beta).
\]
 Hence, in particular,
\[
\inf_{\beta>0} \left|\left\{x:|u_{l,\beta}(x)|>\delta\right\}\right|\geq (2R)^N|B_1|>0.
\]
Moreover, we are in a position to apply Proposition \ref{prop:deadcore} to $|u_{i,\beta}|$ in $B_{2R}(x_\beta)$. 
Indeed, by assumption, 
\[
|u_{i,\beta}| \le A \text{ on }\partial B_{2R}(x_\beta),\qquad\qquad
|u_{i,\beta}| \ge 0 \text{ in } B_{2R}(x_\beta);
\]
additionally, using Kato's inequality (see e.g. 
\cite[Corollary 1.2]{MR2056467}) and recalling assumption \eqref{eq:ass_f}, we obtain
\[
\begin{split}
-\Delta |u_{i,\beta}| &\le |f_i(x,u_{i,\beta})|-\beta |u_{i,\beta}|^p\dsum_{j\neq i}a_{ij} 
|u_{j,\beta_n}|^{p+1}
\leq L_A |u_{i,\beta}|^p - \beta |u_{i,\beta}|^p \cdot a_{il} |u_{l,\beta}|^{p+1}\\
&\leq -\left(\beta \cdot a_{il} 2^{-p-1}{ \delta}^{p+1}-L_A\right) |u_{i,\beta}|^p.
\end{split}
\]
Hence, for $\beta$ such that 
\[ 
\beta \geq \bar\beta:= \frac{2^{p+1}}{a_{il}\,\delta^{p+1}}\left(L_A + \frac{(1+p)\,2^{N(1-p)+p}}{(1-p)^2}\frac{A^{1-p}}{R^2} \right),
\] Proposition \ref{prop:deadcore} implies that 
$u_{i,\beta}\equiv 0$ in $B_{R}(x_\beta)$.
\end{proof}
\begin{proof}[Proof of Theorem \ref{thm:main2}] 
In this case, let $C>0$ be the constant provided by the second part of Theorem \ref{thm:sttz}, 
$\delta<\bar \delta$ and take $x_\beta \in \Omega$, $R>0$, such that $|u_{l,\beta}(x_\beta)|=\frac{\bar \delta+\delta}{2}>\delta$, $B_{2R}(x_\beta)\subset  \Omega$ and
\[
|u_{l,\beta}(x)| > \delta \quad \text{ for } x\in B_{2R}(x_\beta).
\]
Then one can easily conclude as in the proof of Theorem \ref{thm:main1}
\end{proof}

\section{Proof of Lemmas \ref{lemma:uniform_bounds} and \ref{lemma:Linfty_critical}}\label{sec:proof_of_lemmas}

We start proving $L^\infty$ estimates in the Sobolev subcritical case.

\begin{proof}[Proof of Lemma \ref{lemma:uniform_bounds}]
  By Lemma B.3 of \cite{Struwe_book} and the remark made on page 271 therein, we have $u\in L^q(\Omega)$ for every $q$ (this is stated in \cite{Struwe_book} for solutions but also holds, with the same proof, for nonnegative subsolutions). Let us extend $u$ by zero to $\R^N\setminus \Omega$, denoting such extension again by $u$ (which is still a subsolution of the same equation in $\R^N$). Then, by \cite[Theorem 8.17]{GilbargTrudinger_book}, we deduce that $u\in L^\infty(\Omega)$. This information allows us to avoid the use of truncated functions in the following arguments.

\smallbreak

\noindent Step 1 (iteration scheme). Since $u\in H^1_0(\Omega)$, by Sobolev inequality we have $u\in L^{2^*}(\Omega)=L^{2(s_0+1)}(\Omega)$, for $s_0:=\frac{2}{N-2}$, with $\|u\|_{2^*}$ bounded uniformly by a constant only depending on $M$ and $N$.
Assume next that, for some $s\geq s_0$, we have $u\in L^{2(s+1)}(\Omega)$. Using $u^{2s+1}\in H^1_0(\Omega)$ as a test function and denoting by $C_S$ the (square of the) best Sobolev constant of the embedding $H^1\hookrightarrow L^{2^*}$, we have
\begin{align}
C_S\frac{2s+1}{(s+1)^2}\| u\|^{2(s+1)}_{{2^*(s+1)}} & \leq \frac{2s+1}{(s+1)^2}\int_\Omega |\nabla u^{s+1}|^2=(2s+1)\int_\Omega u^{2s}|\nabla u|^2 \nonumber = \int_\Omega \nabla u\cdot \nabla u^{2s+1}\\
& \leq a \int_\Omega (u+u^{q})u^{2s+1}= a(\|u\|_{{2(s+1)}}^{2(s+1)}+ \| u\|^{2s+q+1}_{2s+q+1}). \label{eq:rhs}
\end{align}

Since $1<q<2^*-1$, we have $2s+q+1\in (2(s+1),2^*(s+1))$, and (by H\"older's and Young's inequality)
\[
a\| u\|^{2s+q+1}_{2s+q+1}\leq a\|u\|_{2^*(s+1)}^{\frac{N(q-1)}{2}}\|u\|_{2(s+1)}^\frac{4(s+1)-(q-1)(N-2)}{2}\leq \frac{C_S}{2}\frac{2s+1}{(s+1)^2}\|u\|_{2^*(s+1)}^{2(s+1)} + \kappa_s \|u\|_{2(s+1)}^{2(s+1)\sigma_s},
\]
for some $\kappa_s=\kappa_s(a,N,q)$ and 
\[
\sigma_s:=\frac{4(s+1)-(q-1)(N-2)}{4(s+1)-N(q-1)}>1.
\]
Observe that, in order to apply Young's inequality, we used the fact that
\[
\frac{4(s+1)}{N(q-1)}>\frac{4(s_0+1)}{N(2^*-2)}=1.
\]

Going back to \eqref{eq:rhs}, we deduce that
\[
\frac{C_S}{2}\frac{2s+1}{(s+1)^2}\| u\|^{2(s+1)}_{{2^*(s+1)}}\leq a\|u\|_{2(s+1)}^{2(s+1)}
+\kappa_s \|u\|_{2(s+1)}^{2(s+1)\sigma_s}.
\]
\smallbreak
Considering the increasing, diverging sequence
\[
2(s_0+1)=2^*,\qquad 2(s_{k+1}+1)=2^*(s_{k}+1) \quad  \text{ for } k\in \N_0,
\]
recalling the uniform bound of $\|u\|_{2^*}$, only depending on $M$ and $N$, and by an iterative procedure,
there exists $C_k'=C_k'(a,M,N,q)>0$ such that
\begin{align}\label{eq:iteration_aux}
\|u\|^{2(s_{k}+1)}_{2(s_{k+1}+1)}\leq a \|u\|_{2(s_{k}+1)}^{2(s_{k}+1)} + \kappa_{s_{k}}\|u\|_{2(s_k+1)}^{2(s_k+1)\sigma_{s_k}}\leq C_{k+1}' \quad \text{ for every $k\in \N_0$}.
\end{align}

\noindent Step 2 (conclusion). Fixing $k$ large enough so that
\[
t_k:=\frac{4(s_k+1)}{q}>\max\{N,4\}
\]
and using \cite[Theorem 8.17]{GilbargTrudinger_book}, there exists $C_1$ and $C_2$, both depending on $a,N$ and $t_k$, such that, for any ball $B_2(y)\subset \R^N$,
\begin{align*}
\|u\|_{L^\infty(B_1(y))} &\leq C_1(\|u\|_{L^2(B_2(y))}+\|u+u^q\|_{L^{t_k/2}(B_2(y))})\\
				&\leq C_1(\|u\|_{L^2(\Omega)}+\|u\|_{L^{t_k/2}(\Omega)}+\|u\|^q_{L^{t_kq/2}(\Omega)}) \\
				&\leq C_2(\|u\|_{L^2(\Omega)}+\|u\|_{L^{2(s_k+1)}(\Omega)}+ \|u\|^q_{L^{2(s_k+1)}(\Omega)})\\
				&\leq C_2(M+C_k'+(C_k')^q),
\end{align*}
where, in the last step, we have used the fact that  $2<t_k/2< t_kq/2= 2(s_k+1)$, together with Holder's inequality and \eqref{eq:iteration_aux}. Therefore,
\[
\|u\|_{L^\infty(\Omega)}\leq C_2(M+C_k'+(C_k')^q).\qedhere
\]
\end{proof}

In the Sobolev critical case, the classical Moser/Brezis-Kato iteration scheme  (see for instance \cite[Lemma B.3]{Struwe_book}) does not allow, for nonnegative subsolutions of critical equations, a universal $L^\infty$ bound; the extra assumption that allows this is the strong convergence of the solutions. This observation has been done in \cite[pp. 81-82]{ByeonZhangZou_CVPDE_2013}. 

\begin{proof}[Proof of Lemma \ref{lemma:Linfty_critical}]
Repeating Step 1 of the proof of the previous lemma and using its notation, assuming that $u_n\in L^{2(s+1)}(\Omega)$ for some $s>0$, we obtain
\begin{align}\label{eq:iteration_2^*}
C_S\frac{2s+1}{(s+1)^2}\| u_n\|^{2(s+1)}_{{2^*(s+1)}} \leq a( \|u_n\|_{{2(s+1)}}^{2(s+1)}+ \| u_n\|^{2s + 2^*}_{2s + 2^*}).
\end{align}
For $\kappa_N:=\max\{1,2^\frac{6-N}{N-2}\}$ and $\eta>0$, writing $\Omega=\{u_\infty\leq \eta\}\cup \{u_\infty>\eta\}$ and using H\"older's inequality yields
\begin{align}
\| u_n\|^{2s + 2^*}_{2s + 2^*}&=\int_\Omega |u_n|^{2^*-2}|u_n|^{2(s+1)}\leq \kappa_N \int_\Omega |u_n-u_\infty|^{2^*-2}|u_n|^{2(s+1)}+\kappa_N\int_\Omega |u_\infty|^{2^*-2}|u_n|^{2(s+1)} \nonumber \\
		&\leq \kappa_N \eta^{2^*-2}\|u_n\|_{2(s+1)}^{2(s+1)}+\kappa_N(\|u_n-u_\infty\|_{2^*}^{2^*-2}+\|u_\infty\|_{L^{2^*}(\{u_\infty>\eta\})}^{2^*-2} )\|u_n\|_{2^*(s+1)}^{2(s+1)}. \label{eq:iteration_2^*_2}
\end{align}
Choose $\eta=\eta(a,N,s,u_\infty)>0$ and $\bar n=\bar n(a,s,N)>0$ such that
\[
ak_N \|u_\infty\|_{L^{2^*}( \{u_\infty>\eta\})}^{2^*-2}\leq \frac{C_S}{4}\frac{2s+1}{(s+1)^2}\quad \text{ and } \quad a\kappa_N \|u_n-u_\infty\|_{2^*}^{2^*-2}\leq \frac{C_S}{4}\frac{2s+1}{(s+1)^2} \quad \text{ for } n\geq \bar n.
\]
Combining this with \eqref{eq:iteration_2^*} and \eqref{eq:iteration_2^*_2} provides
\[
\frac{C_S}{2}\frac{2s+1}{(s+1)^2}\|u_n\|_{2^*(s+1)}^{2(s+1)}\leq a\kappa_N \eta^{2^*-2}\|u_n\|_{2(s+1)}^{2(s+1)}.
\]
Considering once again the sequence $(s_k)_k$ such that
\[
2(s_0+1):=2^*,\qquad  2(s_{k}+1):=2^*(s_{k-1}+1),
\]
we see there exists $C_k=C_k(a,N,u_\infty)>0$ such that
\begin{equation}\label{eq:criticalcase_iteration}
\|u_n\|_{2(s_k+1)}\leq C_k\|u_n\|_{2^*}.
\end{equation}
From here, we can reason exactly as in Step 2 of the proof of Lemma \ref{lemma:uniform_bounds}, using \eqref{eq:criticalcase_iteration} instead of \eqref{eq:iteration_aux}.
\end{proof}

\textbf{Data Availability.} Data sharing not applicable to this article as no datasets were generated or analyzed during the current study.

\bigskip

\textbf{Competing Interest.} The authors report there are no competing interests to declare.

\bigskip

\textbf{Acknowledgments.}
The authors would like to thank Alberto Salda\~na and Monica Clapp for pointing out an open problem which led to this project, and for useful discussion.

\bigskip

\textbf{Funding Declaration.}

HT was supported by the Funda\c c\~ao para a Ci\^encia e a Tecnologia (FCT), Portugal, under the projects with reference UIDB/04459/2020 with
DOI identifier \url{https://doi.org/10.54499/UIDP/04459/2020} (CAMGSD), reference 2023.17881.
ICDT with DOI identifier \url{https://doi.org/10.54499/2023.17881.ICDT} (project SHADE) and reference 2023.13921.PEX with DOI identifier
\url{https://doi.org/10.54499/2023.13921.PEX} (project SpectralOPs)

Work partially supported by: the MUR grant Dipartimento di Eccellenza 2023-2027; GV is member of 
the INdAM-GNAMPA group (``Gruppo Nazionale per l'Analisi Matematica, la Probabilit\`a e le loro 
Applicazioni -- Istituto Nazionale di Alta Matematica'').

\small
\bibliographystyle{abbrv}
\bibliography{GHbib}

\medskip
\small
\begin{flushright}
\noindent 
\verb"hugo.n.tavares@tecnico.ulisboa.pt"\\
CAMGSD and Departamento de Matem\'atica, Instituto Superior T\'ecnico\\ 
Av. Rovisco Pais, 1049-001 Lisboa (Portugal)\\
\noindent 
\verb"gianmaria.verzini@polimi.it"\\
Dipartimento di Matematica, Politecnico di Milano\\ 
piazza Leonardo da Vinci 32, 20133 Milano (Italy)\\
\noindent
\verb"junwei.yu@polimi.it"\\
Dipartimento di Matematica, Politecnico di Milano\\ 
piazza Leonardo da Vinci 32, 20133 Milano (Italy)\\
\end{flushright}

\end{document}